%% file: kakeya-proj.tex
\documentclass[11pt]{article}
\usepackage{amsmath,amsthm,amssymb,hyperref, color,fullpage, bm}
\usepackage{blkarray,bigstrut}
\newcommand{\remove}[1]{}

\makeatletter
\newtheorem*{rep@theorem}{\rep@title}
\newcommand{\newreptheorem}[2]{%
\newenvironment{rep#1}[1]{%
 \def\rep@title{#2 \ref{##1}}%
 \begin{rep@theorem}}%
 {\end{rep@theorem}}}
\makeatother

\hypersetup{
	colorlinks=true,
	linkcolor=blue,%
	citecolor=blue,
	linktoc=page
}

\newtheorem{thm}{Theorem}[section]
\newreptheorem{thm}{Theorem}
\newtheorem{claim}[thm]{Claim}
\newtheorem{lem}[thm]{Lemma}
\newreptheorem{lem}{Lemma}
\newtheorem{define}[thm]{Definition}

\newtheorem{conjecture}[thm]{Conjecture}

\newtheorem{fact}[thm]{Fact}

\def\F{{\mathbb{F}}}
\def\D{{\mathbb{D}}}
\def\Q{{\mathbb{Q}}}
\def\Z{{\mathbb{Z}}}
\def\N{{\mathbb{N}}}
\def\R{{\mathbb{R}}}

\def\C{{\mathbb{C}}}

\def\P{{\mathbb{P}}}

\def\cS{{\mathcal S}}

\def\E{{\mathbb E}} 

\def\cN{{\mathcal{N}}}

\def\mweight{\textsf{mweight}}

\def\hZ{\hat{{\mathbb Z}}}

\def\Ind{{\mathbf{1}}}

\def\Gr{\operatorname{Gr}}

\begin{document}

\title{$(n,k)$-Besicovitch sets do not exist in $\Z_p^n$ and $\hZ^n$ for $k\ge 2$}

\author{Manik Dhar\thanks{Department of Mathematics, Massachusetts Institute of Techonology. Email: \texttt{dmanik@mit.edu}. Part of this work was done while this author was a graduate student at the Department of Computer Science, Princeton University where his research was supported by NSF grant DMS-1953807.}}
\date{}

\maketitle

\begin{abstract}
Besicovitch showed that a compact set in $\R^n$ which contains a unit line segment in every direction can have measure $0$. These constructions also work over other metric spaces like the $p$-adics and profinite integers. It is conjectured that it is impossible to construct sets with measure $0$ which contain a unit $2$-disk in every direction in $\R^n$. We prove that over the $p$-adics and profinite integers any set which contains a 2-flat in every direction must have positive measure. The main ingredients are maximal Kakeya estimates for $(\Z/N\Z)^n$ proven in \cite{dhar2022maximal} and adapting Fourier analytic arguments by Oberlin~\cite{Oberlin2005BoundsFK}. In general, we prove $L^{n-1}$ to $L^{n-1}$ estimates for the maximal operator corresponding to $2$-flats. 
\end{abstract} 


\section{Introduction}
A Kakeya set $K\subseteq \R^n$ over the reals is a compact set containing a unit line segment in every direction. Besicovitch~\cite{BesicovitchOnKP} showed that Kakeya sets can have measure zero. We can make the following general definition.

\begin{define}[Besicovitch Sets over $\R^n$]
An $(n,k)$-Besicovitch set $K$ is a compact set in $\R^n$ containing a translate of every $k$-dimensional unit disk such that $K$ has Lebesgue measure zero.
\end{define}

In other words we know that $(n,1)$-Besicovitch sets exist. The main conjecture about them is that $(n,k)$-Besicovitch sets do not exist for $k\ge 2$. So far, we only know this result for $k\ge \log_2(n)$ due to a work of Bourgain~\cite{Bourgain1991BesicovitchTM}. A simpler proof for this result was given by Oberlin~\cite{Oberlin2005BoundsFK} which also proves maximal function estimates for $\log_2 n$ dimensional flats / affine subspaces.

We can also define Besicovitch sets over the $p$-adics $\Z_p^n$ and profinite integers $\hZ^n$.

\begin{define}[Besicovitch Sets over $\Z_p^n$ and $\hZ^n$]
An $(n,k)$-Besicovitch set $K$ over $\Z_p^n$ {\em (or $\hZ^n$)} is a set containing a translate of every $k$-dimensional subspace such that $K$ has  Haar measure zero.
\end{define}
The definition of the set of $k$-dimensional subspaces $\Gr(\F^n,k)$ in these settings is slightly delicate (especially for $\hZ$ which is not a domain). We will give a precise definition later.

It is possible to construct $(n,1)$-Besicovitch sets over $\Z_p^n$ and $\hZ^n$~\cite{hickman2018fourier}. We show that such constructions are impossible for $k\ge 2$ and the rings $\Z_p$ and $\hZ$.

\begin{thm}\label{thm:noBesi}
$(n,k)$-Besicovitch sets do not exist for $k\ge 2$ in $\Z_p^n$ and $\hZ^n$.
\end{thm}

A key related problem is the Kakeya conjecture, which states that compact sets containing a unit line segment in every direction have Minkowski/Hausdorff dimension $n$. This problem is open over the reals for $n>2$.

Wolff~\cite{Wolff99} introduced this problem over finite fields. Let $\F_q$ be the finite field with $q$ elements. A set $S\subseteq \F_q^n$ is called a Kakeya set if it contains a line in every direction. Wolff conjectured that any Kakeya set $S$ in $\F_q^n$ must have size at least $C_n q^n$ where $C_n>0$ only depends on $n$.

The finite field Kakeya conjecture was resolved by Dvir~\cite{Dvir08} using the polynomial method. Subsequent works~\cite{DKSS13,BukhChao21} improved these polynomial method techniques to show that these sets have size at least $q^n/2^{n-1}$ which is known to be tight. The Kakeya conjecture was also posed for the $p$-adics and profinite rationals~\cite{ellenberg2010kakeya,hickman2018fourier}. These problems are linked to studying Kakeya sets over $(\Z/p^\ell\Z)^n$ and $(\Z/N\Z)^n$. Starting from a solution for $N$ square-free~\cite{dhar2021proof} and the breakthrough solution for $N$ prime power~\cite{arsovskiNew} this problem was resolved for general $N$~\cite{DharGeneral}.

We can alternatively consider sets which contain $k$-flats in every direction. Such sets over $\F_q^n$ are known to have size $(1-o(1))q^n$ using suitable generalizations of the polynomial method~\cite{ellenberg2010kakeya,KLSS2011}. If we consider the case of subspaces having large intersections instead of containment then polynomial method arguments do not work and this problem was completely resolved for large fields in \cite{DDL-2,dhar2022linear}. 

The argument in \cite{Bourgain1991BesicovitchTM,Oberlin2005BoundsFK} and \cite{DDL-2,dhar2022linear} both linearly project the set to one lower dimension and then use known Kakeya bounds in the lower dimensional space. The pullback of $k-1$ dimensional planes under the projection are $k$ dimensional planes. This allows us to take estimates for intersections with $k-1$ dimensional planes in the projection to give estimates for intersections with $k$ dimensional planes. In the papers \cite{Bourgain1991BesicovitchTM,Oberlin2005BoundsFK} Fourier Analytic arguments are used and \cite{DDL-2,dhar2022linear} uses simple incidence and/or probabilistic arguments to get stronger bounds for $k$ planes assuming bounds for $k-1$ planes. As finite field Kakeya bounds for lines are known this allows \cite{DDL-2,dhar2022linear} to get bounds for $k$-planes for $k\ge 2$.

If the arguments of \cite{DDL-2,dhar2022linear} could naively generalize to the setting of $\Z_p^n$ (and $\hZ^n$ or $\R$) we would be able to show that sets in $\Z_p^n$ which intersect with $k$ dimensional planes in every direction in normalized Haar measure (in $\Z_p^k$) $\delta^k$ would have normalized Haar Measure (in $\Z_p^n$) $\delta^k$. It is easy to see that this is false by considering a ball of radius $\delta$ in $\Z_p^n$. We instead adapt the arguments of \cite{Bourgain1991BesicovitchTM,Oberlin2005BoundsFK} to prove the following quantitative result.

\begin{thm}\label{thm:quant}
$\delta\in (0,1]$. Given a set $S$ in $\Z_p^n$ (or $\hZ^n$) such that for every $2$-dimensional subspace $U$, a translate of it intersects with $S$ in normalized Haar measure at least $\delta^2$ then $S$ has normalized Haar measure at least $C_{n}\delta^{2(n-1)}$, where $C_n$ is a constant only depending on $n$.
\end{thm}

Over the reals it has been conjectured that the correct measure lower bound is $\delta^n$ (which would mean the radius $\delta$ ball is a tight example). We expect the same to hold over the $p$-adics and profinite integers. The general conjecture is the following.

\begin{conjecture}
Given a set $S$ in $\Z_p^n$ (or $\hZ^n$) such that for every $k$-dimensional subspace $U$, a translate of it intersects with $S$ in normalized Haar measure at least $\delta^k$ then $S$ has normalized Haar measure at least $C_{n}\delta^{n}$, where $C_n$ is a constant only depending on $n$.
\end{conjecture}
The arguments here can also be used to prove bounds for $k$-flats. As we do not get strong enough bounds to prove the conjecture, for simplicity, we stick to the $2$-flats case.

\subsection{Maximal operator bounds for flats}

Let $\Gr(F^n,k)$ be the set of $k$ dimensional subspaces of $\F^n$ ($\F$ in this paper would be $\Z_p$, $\hZ$, and $\Z/N\Z$ and we will give precise definitions later). We define $\cN^k$ as an operator which maps function $f:\F^n\rightarrow \C$ to functions $\Gr(\F^n,k)\rightarrow \C$ as follows,
$$\cN^k f(U)= \sup\limits_{a\in \F^n} \int\limits_{x\in U} |f(a+x)| \,d\mu, $$
where $U\in \Gr(\F^n,k)$ and $\mu$ is the normalized Haar measure over $U\cong \F^{n-1}$ (for $\F=\Z/N\Z$ the Haar measure is just the normalized counting measure). $\cN^k$ for a given subspace $U$ gives you the largest `intersection' a shift of $U$ can have with $f$. 

Our main theorem, which implies Theorems~\ref{thm:noBesi} and \ref{thm:quant}, is the following.

\begin{thm}[$n-1$ to $n-1$ norm bound for the maximal operator $\cN^2$]\label{thm:maxNorm}
Given a function $f:\F^n\rightarrow \C$ where $\F=\Z_p$ or $\hZ$ we have the following bound,
$$C_n\int\limits_{U\in \Gr(\F^n,2)} |\cN^2 f(U)|^{n-1} \,d\nu \le  \int\limits_{x\in \F^n} |f(x)|^{n-1} \,d\mu, $$
where $C_n$ is a constant depending on $n$ and $\F$, and $\mu$ and $\nu$ are normalized Haar measure over $\F^n$ and $\Gr(\F^n,2)$ respectively.
\end{thm} 
The conjecture in this setting is to have $n/2$ to $n$ norm bounds ($n/k$ to $n$ norm bounds for $\cN^k$).

\begin{conjecture}[$n/k$ to $n$ norm bound for the maximal operator $\cN^k$]
Let $k\ge 2$. Given a function $f:\F^n\rightarrow \C$ where $\F=\Z_p$ or $\hZ$ we have the following bound,
$$C_n\int\limits_{U\in \Gr(\F^n,k)} |\cN^k f(U)|^{n} \,d\nu \le  \left(\int_{x\in \F^n} |f(x)|^{n/k} \,d\mu\right)^{k}, $$
where $C_n$ is a constant depending on $n$ and $\F$, and $\mu$ and $\nu$ are normalized Haar measure over $\F^n$ and $\Gr(\F^n,k)$ respectively.
\end{conjecture}

The above conjecture cannot hold for $k=1$ because of the existence of zero measure sets containing line segments in every direction.

\subsection{Proof overview}
We just talk about $\Z_p$ here. Given a function $f:\Z_p^n\rightarrow \R_{\ge 0}$, we project it / take a quotient along direction $u\in \Gr(\Z_p^n,1)$ to get $f_u$ (we integrate along preimages to get the value of the function). $f_u$ as a function over $\Gr(\Z_p^n,1)\times \Z_p^{n-1}$ is referred to as the X-ray transform of $f$. We can show that $\cN^1 f_u$ is the same as $\cN^2 f$ for the subspaces $U\in \Gr(\Z_p^n,2)$ containing $u$.

If we had $\cN^1$ operator bounds over $\Z_p^{n-1}$, we could apply those on $f_u$ for a fixed $u$ and take average over $u$ to complete the proof. The problem is that no good $\cN^1$ bounds hold (because of the existence of $0$ measure Kakeya sets). But we do have $\cN^1$ bounds when the function $f$ is locally constant over balls of fixed radius (from \cite{dhar2022maximal}). If $f_i$ is constant over all balls of radius $p^{-i}$ (recall $\Z_p^n$ is made of $p^{in}$ disjoint balls of radius $p^{-i}$) we have (the term $i^n$ is simplified and not exactly correct but captures the growth rate),
$$\int \cN^1 f_{i,u}^{n-1} \,d\mu \le i^n \int f_{i,u}^{n-1} \,d\mu.$$
We use the above statement by decomposing $f=\sum_{i=0}^\infty f_i$ where $f_i$ is locally constant over balls of radius $p^{-i}$ (this is done by a Littlewood-Paley decomposition where the Fourier transform of $f_i$ is the Fourier transform of $f$ at a fixed scale).
Using a simple Fourier analytic argument we can show,
$$\iint_{u,x} f_{i,u}(x)^{n-1} \,  d\mu d\mu \le p^{-i} \int f^{n-1} d\mu.$$
Combining the previous two equations and adding them up for all $i$ gives us,
$$\left(\int \cN^2 f^{n-1}\, d\mu\right)^{1/n} \le \sum\limits_{i=0}^\infty \left(\iint \cN^1 f_{i,u}^{n-1} \, d\mu d\mu\right)^{1/n} \le \left( \sum\limits_{i=0}^\infty ip^{-i/n}\right) \left(\int f^{n-1} d\mu\right)^{1/n}.$$
As $\sum_{i=0}^\infty ip^{-i/n}<\infty$ we are done.

\subsection{Organization:} In Section~\ref{sec-prelim} we state some preliminaries about the $p$-adics and profinite integers and Fourier analysis over them. In Section~\ref{sec-proofs} we give the proof of the main theorem.



\input{preliminaries}

\input{main_thm.tex}

\bibliographystyle{alpha}
\bibliography{kakeya-proj}

\input{appendix}

\end{document}

%% file: preliminaries.tex
\section{Preliminaries}\label{sec-prelim}

\subsection{Facts about the $p$-adics and profinite rationals}

The elements of the $p$-adic integers $\Z_p$ are infinite series of the form $a_0+a_1p+\hdots+a_np^n+\hdots$. We note $\Z_p$ has a natural $\mod p^\ell$ map which maps $a_0+a_1p+\hdots+a_np^n+\hdots$ to $a_0+a_1p+\hdots+a_{\ell-1}p^{\ell-1}$. For a number $a\in\Z_p$, we define $v_p(a)=1/p^\ell$ where $\ell$ is the largest non-negative integer such that $a \mod p^\ell = 0$ (for $a=0$, $v_p(0)=0$). $v_p(a)$ gives a metric over $\Z_p$.

The ball $B_{p^{-\ell}}(r)$ of radius $p^{-\ell}$ centered at the $r\in Z_p$ is the set $r+p^\ell Z_p$. This means $p^\ell$ radius balls in $\Z_p$ are isomorphic to $\Z/p^\ell\Z$. This means $\Z_p$ is covered by disjoint $p^\ell$ many balls of radius $p^{-\ell}$. 
These same facts hold over $n$-dimensions. The metric over $\Z_p^n$ is the $\ell_\infty$ norm. $B_{p^{-\ell}}(r)$ for $r\in \Z_p^n$ is the set $r+p^\ell \Z_p^n$. $\Z_p^n$ is covered by disjoint $p^{\ell n}$ many balls of radius $p^{-\ell}$. For $x\in \Z_p^n$, we let $v_p(x)$ be the max of the $v_p$ over the coordinates. 
We let $\mu_p$ be the normalized $p$-adic Haar measure over $\Z_p^n$ (that is the Haar measure such that $\mu_p(\Z_p^n)=1$). The key property we need is that $\mu_p(B_{p^{-\ell}}(r))=p^{-\ell n}$ for any $\ell\ge 0$.

We let $\cS_{p^{-\ell}}(\Z_p^n)$ be functions over $\Z_p^n$ of the form 
$$\sum\limits_{r\in (\Z/p^\ell\Z)^n} a_r \Ind[B_{p^{-\ell}}(r)],$$
for $a_r\in \C$. We let $\cS(\Z_p^n)=\bigcup_{\ell=0}^{\infty} \cS_{p^{-\ell}}(\Z_p^n)$. By the Stone-Weierstrass theorem we know that $\cS(\Z_p^n)$ is dense in the space of $L_q(\Z_p^n)$ functions for all $\infty>q\ge 1$.

We also need to work with the dual group of $\Z_p$ which is $\Z[1/p]/\Z$. Elements in this group are represented by finite sums $\sum\limits_{i=J}^{-1} a_{i} p^i$, the group operation is by addition and $1=0$. Any element $g$ in the dual group gives a map from $\Z_p$ to $\Q/\Z$ by multiplication. To be precise if we write $g=a/p^J$ with $0\le a\in \Z$, $g$ maps $x$ to $(ax \mod p^J)/p^J$ in $\Q/\Z$. In general the dual of $\Z_p^n$ is $(\Z[1/p]/\Z)^n$ and each element $g$ in the dual group gives a map from $x\in \Z_p^n$ to $\langle x,g\rangle\in \Q/\Z$. This gives us a character of $\Z_p^n$ mapping $x$ to $e^{2\pi i \langle x,g\rangle}$. For $x\in (\Z[1/p]/\Z)^n$ we let $v_p(x)=M$ where $M$ is the smallest non-negative power $p^\ell$ of $p$ such that $p^\ell v = 0$ ($p^\ell v$ is just $v$ added to itself $p^\ell$ times). The group $(\Z[1/p]/\Z)^n$ is countable in size and has the discrete topology.

The profinite integers $\hZ$ is simply the product of $\Z_p$ for all prime $p$ with the product topology (and hence metrizable). It is possible to write a general element in $\hZ$ as an infinite series $a_11!+a_22!+\hdots+a_i i!+\hdots$ where $a_i\in\{0,\hdots,i-1\}$ (the equivalence between this and the product of $\Z_p$ can be shown using the Chinese remainder theorem). We can define a mod $N$ operation by mapping $\sum_{j=1}^\infty a_j j!$ to $\sum_{j=1}^{N-1} a_j j! \mod N$. For $a\in \hZ$, let $\hat{v}(a)=1/i$ where $i>0$ is the largest number such that $a$ mod $i$ is $0$ (we set $\hat{v}(0)=0$). We are taking a reciprocal to define $\hat{v}$ because we want to suggest that smaller $\hat{v}$ implies a smaller number\footnote{Indeed, $\hat{v}$ does give us scales but the scales are not naively ordered. To be precise $\hat{v}$ is ordered by divisibility of $1/\hat{v}$ with multiples being smaller, which means the ordering is a poset and not linear. A metric can be obtained by restricting to scales of the form $i!$ for some $i$.}.  

For a given positive integer $N$ and $x\in \hZ$, let $B_{1/N}(x)$ be the set of $y\in \hZ$ such that $N$ divides $1/\hat{v}(x-y)$.

Similar to the $p$-adic case we see that $\hZ$ can be covered by $N$ many disjoint balls $B_{1/N}(r)$ where $r$ can be thought to be from $\Z/N\Z$. 
For $x\in\hZ^n$ we let $\hat{v}(x_1,\hdots,x_n)=1/\operatorname{gcd}(1/\hat{v}(x_1),\hdots,1/\hat{v}(x_n))$. We similarly define $B_{1/N}$ for $\hZ^n$ and again $\hZ^n$ can be covered by $N^n$ many disjoint balls $B_{1/N}(r)$ where $r$ can be thought to be from $r\in (\Z/N\Z)^n$. Let $\mu_{\hZ^n}$ be the normalized Haar measure over $\hZ^n$ then we have $\mu_{\hZ^n}(B_{1/N}(x))=1/N^n$ for any $x\in \hZ^n$. 

We let $\cS_{1/N}(\Z_p^n)$ be functions over $\hZ^n$ of the form 
$$\sum\limits_{r\in (\Z/N\Z)^n} a_r \Ind[B_{1/N}(r)],$$
for $a_r\in \C$. We let $\cS(\hZ^n)=\bigcup_{N=1}^{\infty} \cS_{1/N}(\hZ^n)$, by the Stone-Weierstrass theorem we have that $\cS(\hZ^n)$ is dense in the space of $L_q(\hZ^n)$ functions for all $\infty>q\ge 1$.

As $\hZ$ is the product of all $\Z_p$, the dual of $\hZ$ is the product of all $\Z[1/p]/\Z$ which is $\Q/\Z$. Any element in $\Q/\Z$ can be written as $a/b$ with $0\le a<b$. Any element $a/b\in \Q/\Z$ gives a function from $\hZ$ to $\Q/\Z$ by mapping $x\in \hZ$ to $(ax \mod b)/b$. In general $(\Q/\Z)^n$ is the dual of $\hZ^n$ and an element $g\in (\Q/\Z)^n$ gives a map of $x\in \hZ^n$ to $\langle x,g\rangle\in \Q/\Z$. This means $g\in (\Q/\Z)^n$ gives us the character $x$ to $e^{2\pi i \langle x, g\rangle}$. $(\Q/\Z)^n$ is countable in size and has the discrete topology. For a given $g\in (\Q/\Z)^n$ we let $\hat{v}(\Q/\Z)=N$ where $N$ is the smallest positive number such that $Ng=0$. 

Next we give a definition of subspaces in $\Z_p^n$ and $\hZ^n$.

\begin{define}[Grassmanian in $\Z_p^n$ and $\hZ^n$]
The Grassmanian $\Gr(\Z_p^n,k)$ of $k$-dimensional subspaces in $\Z_p^n$ is the set of sub-modules generated by the rows of a $k\times n$ matrix such that one of its $k\times k$ minors is a unit in $\Z_p$. 

The Grassmanian $\Gr(\hZ^n,k)$ of $k$-dimensional subspaces in $\hZ^n$ is the set of sub-modules generated by the rows of a $k\times n$ matrix such that for every prime $p$ there is one $k\times k$ minor which is non zero mod $p$ (in other words the $\hat{v}$ of the Pl\"ucker coordinates is $1$). 
\end{define}

We also define them over $(\Z/N\Z)^n$.

\begin{define}[Grassmanian in $(\Z/N\Z)^n$]
The Grassmanian $\Gr((\Z/N\Z)^n,k)$ is the set of sub-modules generated by the rows of a $k\times n$ matrix $M$ such that for every prime $p$ which divides $N$ there is one $k\times k$ minor of $M$ which is non zero mod $p$. 
\end{define}

We use $\P \F^{n-1}$ to refer to $\Gr(\F^n,1)$ and we can identify $\P \F^{n-1}$ as a subset of $\F^n$ by picking representatives. The Haar measure on $\F$ also induces a measure on $\Gr(\F^n,k)$.

We note by the Chinese Remainder Theorem, if $N=q_1\hdots q_r$ where $q_i$ are powers of distinct primes, $\Gr((\Z/N\Z)^n,k)=\Gr((\Z/q_1\Z)^n,k)\times \hdots\times \Gr((\Z/q_r\Z)^n,k).$

For a given $u\in \P \F^{n-1}$, we let $Q_u$ be the quotient $\F^n/\langle u \rangle$ which is isomorphic to $\F^{n-1}$. We choose representative vectors in $\P Q_u$ such that they are a subset of $\P \F^{n-1}$. 

\subsection{Fourier Analysis over $\Z_p$ and $\hZ$}

To adapt the Fourier analytic argument of Oberlin~\cite{Oberlin2005BoundsFK} we cover some 
   simple facts about Fourier Analysis over $\Z/N\Z$. Let $e$ be the map from $\Q/\Z$ to $\C$ which maps $x\in \Q/\Z$ to $e^{2\pi i x}\in \C$. Given $u\in \P \Z_p^{n-1}$ ,we let $u^\perp=\{a\in (\Z[1/p]/\Z)^n| \langle u,a\rangle = 0\}.$ We similarly define $u^\perp$ for $u \in \hZ^{n-1}$.

We will need a simple lemma about the density of vectors in $\P \F^{n-1}$ orthogonal to a fixed $v\in \D^n$ where $\D$ is the dual of $\F$. 



\begin{lem}\label{lem:radiusN}
$\F=\Z_p,\hZ$ and $\D=\Z[1/p]/\Z, \Q/\Z$ respectively. For a given $a\in \D^n$, the set of $u\in \P \F^{n-1}$ such that $a\in u^\perp$ has measure
$$ \frac{|\P (\Z/N\Z)^{n-2}|}{|\P (\Z/N\Z)^{n-1}|},$$
where $N=v(a)$ {\em ($v=\hat{v}$ for $\F=\hZ$ and $v=v_p$ for $\F=\Z_p$).} 
\end{lem}
\begin{proof}
We give the proof for $\F=\hZ$ and $\D=\Q/\Z$. The proof is the same for $\Z_p$ and $\Z[1/p]/\Z$ (with $v_p$ used instead of $\hat{v}$). 

For $a\in (\Q/\Z)^n$, say $\hat{v}(a)=N$. This means $a=a'/N$ where $a'\in \{0,\hdots,N-1\}^n$ and $ta' \mod N\ne 0$ for $t=1,\hdots,N-1$. We also use $a'$ to refer to $a' \mod N \in (\Z/N\Z)^n$. We now see that for any $u \in \P \hZ^{n-1}, \langle u,a\rangle = 0$ if and only if  $\langle u \mod N, a' \rangle = 0\in \Z/N\Z$. For $u\in \P (\Z/N\Z)^{n-1}$, $u + N \hZ^n$ partitions $\P \hZ^{n-1}$ into  $|(\Z/N\Z)^{n-1}|$ many disjoint isomorphic pieces. Therefore $u + N \hZ^n\subseteq \P \hZ^{n-1}$ has measure $1/|\P (\Z/N\Z)^{n-1}|$.

As $ta' \mod N\ne 0$ for $t\in \{0,\hdots,N-1\}$ then the set of $b\in (\Z/N\Z)^n$ such that $\langle b,a'\rangle =0$ is isomorphic to $(\Z/N\Z)^{n-1}$. This means that the number of $u \in \P (\Z/N\Z)^{n-1}$ such that $\langle u, a'\rangle =0$ is $|\P (\Z/N\Z)^{n-2}|$. 
\end{proof}

\begin{define}[Fourier transform over $\Z_p$ and $\hZ$]
$\F=\Z_p,\hZ$ and $\D=\Z[1/p]/\Z, \Q/\Z$ respectively. Given a $f:\F^n\rightarrow \C$, the Fourier transform $\hat{f}: \D^n\rightarrow \C$ is defined as
$$\hat{f}(a)= \int\limits_{x\in \F^n} e(\langle x,u\rangle)f(x)\, d\mu,$$
where $\mu$ is the normalized Haar measure on $\F^n$. The inverse of the above is,
$$f(x)= \sum\limits_{a\in \D^n} e(-\langle x,v\rangle)\hat{f}(a).$$
\end{define}

We have Plancherel's Theorem
$$\int_{x\in \F^n} |f(x)|^2\, d\mu  = \sum\limits_{a\in \D^n} |\hat{f}(a)|^2,$$
where $\F=\Z_p,\hZ$ and $\D=\Z[1/p]/\Z, \Q/\Z$ respectively.

We also need to define the X-ray transform which looks at the pushforward of a function on quotients along every direction.

\begin{define}[X-ray transform]
Given a $f:\F^n\rightarrow \C$ (where $\F=\Z_p,\hZ,\Z/N\Z$) we define $f_u: Q_u \rightarrow \C$ for $u\in \P \F^{n-1}$ as,
$$f_u(x)=\int_{t\in \F} f(x+tu) \,d\mu,$$
where $\mu$ is the normalized Haar measure over $\F^n$.
\end{define}

We next prove a key lemma which will connect the $\ell^2$ norm of the X-ray transform with a Sobolev norm of $f$ (analogous to Lemma 3.1 in \cite{Oberlin2005BoundsFK}).

\begin{lem}
Given $f:\F^n\rightarrow \C$ {\em (where $\F=\Z_p$ or $\hZ$, $\D=\Z[1/p]/\Z, \Q/\Z$, and $v=v_p,\hat{v}$ respectively) },

$$\int\limits_{u\in \P \F^{n-1}} \int\limits_{x\in Q_u} |f_u(x)|^2\, d\mu \,d\nu=\sum\limits_{a\in \D^n} \frac{|\P (\Z/v(a)\Z)^{n-2}|}{|\P (\Z/v(a)\Z)^{n-1}|} |\hat{f}(a)|^2,$$
where $\mu$ is the normalized Haar measure on $\F^{n-1}$ and $\nu$ is the normalized Haar measure on $\P \F^{n-1}$.
\end{lem}
\begin{proof}
As usual $\F=\Z_p,\hZ$ and $\D=\Z[1/p]/\Z, \Q/\Z$ respectively.
\begin{claim}
For $u\in \P \F^{n-1}$ we have,
$$\int\limits_{x\in Q_u} |f_u(x)|^2 \,d\mu=\sum\limits_{a\in u^\perp} |\hat{f}(a)|^2. $$
\end{claim}
\begin{proof}
We prove this claim for $f\in \cS(\F^n)$ which suffices as these functions are dense in the space of $L_2$ functions over $\F^n$. Let $f\in \cS_{1/N}(\F^n)$ ($N$ will be a power of $p$ for $\F=\Z_p$ and a positive integer for $\F=\hZ$).
\begin{align*}
\sum\limits_{a\in u^\perp} |\hat{f}(a)|^2&=\sum\limits_{a\in u^\perp}\iint_{x,y\in \F^n} e(\langle x-y,a\rangle)f(x)\overline{f(y)}\, d\mu d\mu\\
&=N^{-2n}\sum\limits_{a\in u^\perp}\sum\limits_{x,y\in (\Z/N\Z)^n}\iint_{x',y'\in \F^n} e(\langle x-y+N(x'-y'),a\rangle)f(x)\overline{f(y)}\, d\mu d\mu,
\end{align*}
as $f$ is constant over $x+N\F^n$. We see that if $Na\ne 0$ then $\iint_{x',y'\in \F^n} e(\langle x-y+N(x'-y'),a\rangle)=0$ and otherwise it equals $e(\langle x-y,a\rangle)$. The set of $a\in u^\perp$ such that $Na=0$ is finite. We therefore get,
$$\sum\limits_{a\in u^\perp} |\hat{f}(a)|^2= N^{-2n}\sum\limits_{a\in u^\perp, Na=0}\sum\limits_{x,y\in (\Z/N\Z)^n}e(\langle x-y,a\rangle)f(x)\overline{f(y)}.$$


We see that if $x-y$ is not a multiple of $u$ then $\sum\limits_{a\in u^\perp, Na=0} e(\langle x-y,v\rangle)=0$ and if $x-y$ is a multiple of $u$ then $\sum_{a\in u^\perp, Na=0} e(\langle x-y,a\rangle)=N^{n-1}$. Using this in the above equation gives us,
\begin{equation*}\label{eq-xnorm1}
\sum\limits_{a\in u^\perp} |\hat{f}(a)|^2=N^{-n-1}\sum\limits_{x\in (\Z/N\Z)^n}\sum\limits_{t\in \Z/N\Z} f(x)\overline{f(x+tu)}=\iint\limits_{x\in \F^n,t\in \F} f(z)\overline{f(z+tu)}.
\end{equation*}
We also have,
$$\int\limits_{x\in Q_u} |f_u(x)|^2\,d\mu=\iiint\limits_{x\in Q_u,t_1,t_2\in \F} f(x+t_1u)\overline{f(x+t_2u)}=\iint\limits_{x\in \F^n,t\in \F} f(z)\overline{f(z+tu)}.$$
\end{proof}

Using the above claim we have,
\begin{equation*}\label{eq-xnorm2}
\int\limits_{u\in \P \F^{n-1}} \int\limits_{x\in Q_u} |f_u(x)|^2 \,d\mu\,d\nu=\int\limits_{u\in \P \F^{n-1}}\sum\limits_{v\in u^\perp} |\hat{f}(v)|^2.
\end{equation*}

To simplify the above sum for a given $v\in \D^n$, we want the measure of the set of $u\in \P \F^{n-1}$ such that $v\in u^\perp$. Lemma~\ref{lem:radiusN} does exactly that completing the proof.
\end{proof}

We want to use complex interpolation with the above lemma. We will use the Riesz-Thorin theorem and we state it for the case we need.

\begin{thm}[Riesz-Thorin interpolation (See Chapter 4 in \cite{katznelson_2004})]
Let $W$ be a linear operator mapping complex valued functions $f$ over $\F^n$ (where $\F=\hZ$ or $\Z_p$) to complex valued functions over some other compact measurable domain $G$ with measure $\nu$ (for us $G=\P \F^{n-1}\times \F^n$). If we have 
$$ \int\limits_{y\in G} |Wf(y)|^{p_1}\, d\nu\le B\int\limits_{y\in \F^n} |f(x)|^{p_1}\, d\mu,$$ 
and 
$$ \sup\limits_{y\in G'} |Wf(y)| \le \sup\limits_{x\in \F^n} |f(x)|$$
for some constant $B>0$ and $1\le p< \infty$ we have
$$ \int\limits_{y\in G} |Wf(y)|^{p_\theta}\, d\nu\le B \int\limits_{y\in \F^n} |f(x)|^{p_\theta}\, d\mu,$$
where $p_\theta = p/\theta$ for any $\theta \in [0,1]$.
\end{thm}

\begin{lem}\label{lem:freqBound}
$\F=\Z_p$ or $\hZ$ and $\D=\Z[1/p]/\Z$ or $\Q/\Z$ respectively. Let $f:\F^n \rightarrow \C$ we have the following inequality,
$$\int\limits_{u\in \P \F^{n-1}} \int\limits_{x\in Q_u} |f_{M_1,M_2,u}(x)|^p\, d\mu \,d\nu\le \frac{|\P (\Z/M_1\Z)^{n-2}|}{|\P (\Z/M_1\Z)^{n-1}|} \int\limits_{x\in \F^n} |f(x)|^p\, d\mu,$$
for all $p\ge 2$, $M_1<M_2$ are positive integers for $\F=\hZ$ and powers of $p$ for $\F=\Z_p$, and $f_{M_1,M_2,u}$ is the X-ray transform of 
$$f_{M_1,M_2}(x) = \sum\limits_{a\in \D^n, M_1\le v(a)<M_2} e(-\langle x,v\rangle)\hat{f}(a).$$
\end{lem}
\begin{proof}
The X-ray transform is a linear map on functions over $\F^n$ to functions over $\P \F^{n-1}\times \F^n$. Going from $f$ to $f_{M_1,M_2}$ is also a linear map. Therefore mapping $f$ to $f_{M_1,M_2,u}: \P \F^{n-1}\times \F^{n-1}\rightarrow \F^n$ is a linear map. Lemma~\ref{lem:radiusN} and Plancherel's theorem gives us,
$$\int\limits_{u\in \P \F^{n-1}} \int\limits_{x\in Q_u} |f_{M,u}(x)|^2\, d\mu \,d\nu\le \frac{|\P (\Z/M_1\Z)^{n-2}|}{|\P (\Z/M_1\Z)^{n-1}|} \int\limits_{x\in \F^n} |f(x)|^2\, d\mu.$$
By the triangle inequality also have 
$$\sup\limits_{u\in \P \F^{n-1}, x\in Q_u\cong \F^{n-1}} |f_{M,u}(x)| \le \sup\limits_{x\in \F^n} |f(x)|.$$ 
Riesz-Thorin interpolation now gives us the result.
\end{proof}

\subsection{Maximal bounds for $\cN^1$ over $(\Z/N\Z)^n$}

We will also need to use Maximal bounds for $\cN^1$ over $(\Z/N\Z)^n$ from \cite{dhar2022maximal}. The version stated below is derived using a simple argument from the statement in \cite{dhar2022maximal} which we give in the appendix.

\begin{thm}[Maximal Kakeya bounds over $\Z/N\Z$ for general $N$]\label{thm-maxEst}
Let $n,N>0$ be integers. For any function $f:(\Z/N\Z)^n\rightarrow \C$ we have the following bound,
$$\underset{x \in (\Z/N\Z)^{n}}{\E} |f(x)|^n \ge  C_{N,n}\underset{u \in \P (\Z/N\Z)^{n-1}}{\E}[|\cN^1 f(u)|^n],$$
where
{\em \begin{align*}
    C_{N,n}=N^{-C n\log(N)/\log \log N},
\end{align*}} and $C>0$ is a universal constant.
\end{thm}




%% file: main_thm.tex
\section{Proof of Theorem~\ref{thm:maxNorm}}\label{sec-proofs}
Throughout $\F=\Z_p,\hZ, \D=\Z[1/p]/\Z,\Q/\Z$, and $v=v_p,\hat{v}$. Without loss of generality we can assume $f$ is a function with outputs in the non-negative reals. We let $M_0,M_1,\hdots,M_j,\hdots$ be a sequence of scales. For $\F=\Z_p$ we set $M_i=p^i$ and $\F=\hZ$ we set $M_i=(i+1)!$. Given any function $f: \F^n \rightarrow \R_{\ge 0}$ and for each scale $M_i$ we let 

$$f_i(x)=\sum\limits_{a\in \D^n, M_i\le v(a) < M_{i+1}} e(-\langle x,v\rangle)\hat{f}(a).$$
$f_i$ are the Littlewood-Paley decomposition of $f$.
We note that $f_i$ is constant over cosets of $M_{i+1}\F^n$, this means $f_i$ induces a function $f'_i:(\Z/M_{i+1}\Z)^n\rightarrow \R_{\ge 0}$. 

We now consider the X-ray transform $f_{i,u}:Q_u\rightarrow \R_{\ge 0}$ of $f_i$, where $u\in \P \F^{n-1}$. Consider a $w\in \P Q_u$ and $U\in \Gr(\F^n,2)$ such that $u,w\in U$. $u,w\in U$ and $w\in \P Q_u$  implies that $u$ and $w$ span $U$. For a fixed $u$ and $U$ containing $u$, $w\in  \P Q_u$ has to be unique and always exists. We then have 
\begin{equation}\label{eq:projMaxPlane2Line}
\cN^2 f_i (U) = \cN^1 f_{i,u}(w). 
\end{equation}
Let $Q_{i,u}$ be $Q_u \mod M_{i+1}$  (quotient by $M_{i+1}Q_u$). $Q_u$ is isomorphic to $\F^{n-1}$ and as before $f_{i,u}$ is constant over cosets of $M_{i+1}Q_u$ so $f_{i,u}$ induces a function $f'_{i,u}$ over $Q_{i,u}\cong (\Z/M_{i+1}\Z)^{n-1}$. It easily follows that $f'_{i,u}$ is the X-ray transform of $f_{i,u}$. We thus have,
\begin{equation}\label{eq:divisorMaxReduction}
\int\limits_{w \in \P Q_u} \cN^1 f_{i,u}(w)^{n-1} \,d\nu =\underset{w \in \P Q_{i,u}}{\E}[\cN^1 f'_{i,u}(w)^{n-1}],
\end{equation}
where $\nu$ is the normalized Haar measure on $\P Q_u\cong \P \F^{n-1}$. If we apply Theorem~\ref{thm-maxEst} on $f'_{i,u}: Q_{i,u}\cong (\Z/M_{i+1}\Z)^{n-1}\rightarrow \R_{\ge 0}$ we have,
\begin{equation*}
\underset{v \in \P Q_{i,u}}{\E}[\cN^1 f'_{i,u}(v)^{n-1}] C_{M_{i+1},n-1} \le \underset{x \in Q_{i,u}}{\E}  f'_{i,u}(x)^{n-1}=\int\limits_{x\in Q_u}f_{i,u}(x)^{n-1}\, d\mu,
\end{equation*}
where $C_{M_{i+1},n-1}=M_{i+1}^{-C (n-1)\log(M_{i+1})/\log \log M_{i+1}}$ and $C>0$ is a universal constant.

In the above equation applying \eqref{eq:divisorMaxReduction}, taking expectation over $u\in \P \F^{n-1}$, and finally applying \eqref{eq:projMaxPlane2Line} gives us,

\begin{equation*}
C_{M_{i+1},n-1}\int\limits_{ U \in \Gr(\F^n,2)} \cN^2 f_{i}(U)^{n-1} \, d\nu_2 \le  \iint\limits_{u \in \P \F^{n-1},x\in Q_u} f_{i,u}(x)^{n-1} \, d\nu d\mu,
\end{equation*}
where $\nu_2$ is the normalized Haar measure over $\Gr(\F^n,2)$. Applying Lemma~\ref{lem:freqBound} with the previous equation gives us,
\begin{equation}\label{eq-final1}
    C_{M_{i+1},n-1}\int\limits_{ U \in \Gr(\F^n,2)} \cN^2 f_{i}(U)^{n-1} \, d\nu_2 \le  {|\P (\Z/M_i\Z)^{n-2}| \over |\P (\Z/M_i\Z)^{n-1}|} \int\limits_{x\in \F^n} f(x)^{n-1} \, d\mu.
\end{equation}
Using the triangle inequality for the $n-1$ norm and the fact that $\cN^2 f + \cN^2 g \ge \cN^2 (f+g)$ we have,
\begin{align*}\label{eq-final2}\sum\limits_{i=0}^\infty \left(\int_{ U \in \Gr(\F^n,2)} \cN^2 f_{i}(U)^{n-1}\right)^{1/(n-1)} &\ge  \left(\int_{ U \in \Gr(\F^n,2)} \left(\sum\limits_{i=0}^\infty\cN^2 f_{i}(U)\right)^{n-1} \right)^{1/(n-1)}\\
&\ge \left(\int_{ U \in \Gr(\F^n,2)} \cN^2 f(U)^{n-1}\right)^{1/(n-1)}.
\end{align*}
Using \eqref{eq-final1} and the previous equation gives us,

$$\int\limits_{ U \in \Gr(\F^n,2)} \cN^2 f(U)^{n-1} \le \left(\sum\limits_{i=0}^\infty {|\P (\Z/M_i\Z)^{n-2}|^{1/(n-1)} \over C_{M_{i+1},n-1}^{1/(n-1)}|\P (\Z/M_i\Z)^{n-1}|^{1/(n-1)}}\right)^{n-1} \int\limits_{x\in \F^n} f(x)^{n-1} \, d\mu. $$

The following claim completes the proof,
\begin{claim}
 $$  \sum\limits_{i=0}^\infty   {|\P (\Z/M_i\Z)^{n-2}|^{1/(n-1)} \over C_{M_{i+1},n-1}^{1/(n-1)}|\P (\Z/M_i\Z)^{n-1}|^{1/(n-1)}}<\infty $$
\end{claim}
\begin{proof}
We need the size of $\P (\Z/N\Z)^n$ for general $N=p_1^{r_1}\hdots p_t^{r_t}$ where $p_1<\hdots<p_t$ are distinct primes.
\begin{equation*} 
|\P (\Z/N\Z)^{n-1}| = \prod\limits_{j=0}^t\frac{p_j^{r_jn}-p_j^{(r_j-1)n}}{p_j^{r_j}-p_j^{r_j-1}}.
\end{equation*}
This implies,
$${|\P (\Z/N\Z)^{n-2}| \over |\P (\Z/N\Z)^{n-1}|}= \prod\limits_{j=0}^t\frac{p_j^{r_j(n-1)}-p_j^{(r_j-1)(n-1)}}{p_j^{r_jn}-p_j^{(r_j-1)n}}\le \prod\limits_{j=0}^t{\frac{1}{p_j^{r_j}}}\le \frac{1}{N}.$$
Finally we have,
$$\sum\limits_{i=0}^\infty   {|\P (\Z/M_i\Z)^{n-2}|^{1/(n-1)} \over C_{M_{i+1},n-1}^{1/(n-1)}|\P (\Z/M_i\Z)^{n-1}|^{1/(n-1)}}\le \sum\limits_{i=0}^\infty {M_{i+1}^{C\log M_{i+1}/\log \log M_{i+1}} \over M_i^{1/(n-1)}}.$$
The sum above is clearly bounded for $M_i=p^i$ or $M_i=(i+1)!$
\end{proof}

%% file: appendix.tex
\appendix

\section{Proof of Theorem~\ref{thm-maxEst}}

Theorem~\ref{thm-maxEst} easily follows from the results in \cite{dhar2022maximal}. To state the result from \cite{dhar2022maximal} we first need some simple facts which follow from the Chinese remainder theorem.

\begin{fact}[Geometry of $\Z/p^kN_0\Z$]\label{fact:geoChine}
Let $p,N_0,n,k\in \N,R=\Z/p^kN_0\Z,R_0=\Z/N_0\Z$ with $p$ prime and co-prime to $N_0$.  Using the Chinese remainder theorem we know that any co-ordinate in $R^n$ can be uniquely represented by a tuple in $(\Z/p^k\Z)^n\times R_0^n$. Also any direction in the projective space $\P R^{n-1}$ can again be uniquely represented by a tuple in $\P (\Z/p^k\Z)^{n-1} \times \P R_0^{n-1}$. Finally, any line $L$ with direction $b=(b_p,b_0)\in \P (\Z/p^k\Z)^{n-1} \times \P R_0^{n-1}$ in $R^n$ is equivalent to the product of a line $L_p\subset (\Z/p^k\Z)^n$ in direction $b_p$ and a line $L_0\subset R_0^n$ in direction $b_0$. 
\end{fact}

\begin{define}[$p$-Maximal weight]
For $p$ and $N$ coprime and $f:(\Z/p^kN\Z)^n\rightarrow \N$ we define the $p$-maximal weight {\em $\mweight(f,p)$} as follows:

Let $L(u)=\{a_u+tu|t\in \Z/p^kN\Z\}$ be a line such that $\sum_{x\in L(u)} f(x)=f^*(u)$. Using Fact~\ref{fact:geoChine} we note that the line $L(u)$ can be written as a product of lines $L_p(u)\subseteq (\Z/p^k\Z)^n$ and $L_1(u)\subseteq (\Z/N\Z)^n$. We define,
{\em $$\mweight(f,p)=\sup_{u\in \P (\Z/p^kN\Z)^{n-1},z\in L_1(u)} \sum\limits_{x\in L_p(u)} f((x,z)).$$}
\end{define}

Note, $\mweight(f,p)$ for $f:(\Z/p^k\Z)^n\rightarrow \N$ is simply $\max_u f^*(u)$.

We can finally state the main theorem from \cite{dhar2022maximal}.
We also let 
$$f^*(U)= N \cN^1 f(U)= \sup\limits_{a\in \F^n} \sum\limits_{x\in U} |f(a+x)|.$$

\begin{thm}[Maximal Kakeya bounds over $\Z/N\Z$ for general $N$]\label{thm-maxN}
Let $n>0$ be an integer and $N=p_1^{k_1}\hdots p_r^{k_r}$ with $p_i$ primes and $k_i\in \N$. For any function $f:(\Z/N\Z)^n\rightarrow \N$ we have the following bound,
$$\sum\limits_{x\in (\Z/N\Z)^n} |f(x)|^n \ge  C_{N,n}\underset{u \in \P (\Z/N\Z)^{n-1}}{\E}[|f^*(u)|^n]=\frac{C_{N,n}}{|\P(\Z/N\Z)^{n-1}|} \left(\sum\limits_{u\in \P (\Z/N\Z)^{n-1}} |f^*(u)|^n\right),$$
where
{\em \begin{align*}
    C_{N,n}=&\left(\frac{1}{2(\log(\mweight(f,p_1))+1)\lceil \log_{p_1}(\mweight(f,p_1))+\log_{p_1}(n)\rceil}\right)^n\\
    &\cdot \left(\frac{1}{2(k_r+\lceil \log_{p_r}(n)\rceil)}  \prod\limits_{i=2}^{r-1}\frac{1}{2(k_i\log(p_i)+1)(k_i+\lceil \log_{p_i}(n)\rceil)}\right)^n
\end{align*}.}
\end{thm}

\begin{proof}[Proof of Theorem~\ref{thm-maxEst}]

By a scaling and limiting argument we see that it suffices to prove the theorem for $f: (\Z/N\Z)^n\rightarrow \Q_{\ge 0}$ and $\sum_{x\in (\Z/N\Z)^n} |f(x)|^n = 1$.
In this case we see that $0\le f(x)\le 1$. We divide $[0,1]$ into $N$ parts of length $1/N$. Let $g(x)$ be the function obtained by rounding the value of $f(x)$ to $\lceil N f(x)\rceil /N$. We see that $g(x)\ge f(x)$ which implies $g^*(u)\ge f^*(u)$ for all $u\in \P (\Z/N\Z)^{n-1}$. We also see that if $f(x)\ge 1/(2N)$ then $g(x)\le 2f(x)$. Let $S$ be the set of values of $x$ for which $f(x)\le 1/(2N)$. We see that $\sum_{x\in S} g(x)^n\le 1=\sum_{x \in (\Z/N\Z)^n} f(x)^n$. This gives us,
$$\sum\limits_{x \in (\Z/N\Z)^n} g(x)^n=\sum_{x\in S} g(x)^n+\sum_{x\in S^c} g(x)^n\le (2^n+1)\sum_{x \in (\Z/N\Z)^n} f(x)^n.$$

We note that $Ng$ is a function which takes values in $\{0,1,\hdots,N\}$. Hence, we can apply Theorem~\ref{thm-maxN} (note, $\mweight(Ng,p_1)\le Np\le N^2$).
\begin{align*}
(2^n+1)\sum_{x \in (\Z/N\Z)^n} f(x)^n&\ge N^{-n}\sum\limits_{x \in (\Z/N\Z)^n} (Ng(x))^n\\
&\ge D_{N,n}\underset{u \in \P (\Z/N\Z)^{n-1}}{\E}[|g^*(u)|^n]\ge D_{N,n}\underset{u \in \P (\Z/N\Z)^{n-1}}{\E}[|f^*(u)|^n],
\end{align*}
where
\begin{align*}
D_{N,n}=&\left(\frac{1}{2(2\log N+1)(2\log_{p_1}N+\log_{p_1} n+1)}\right)^n\times \\
&\left(\frac{1}{2(k_r+\lceil \log_{p_r}(n)\rceil)}  \prod\limits_{i=2}^{r-1}\frac{1}{2(k_i\log(p_i)+1)(k_i+\lceil \log_{p_i}(n)\rceil)}\right)^n.
\end{align*}

The expression $\prod_{i=1}^{r} k_i$ is upper bounded by the number of divisors $\tau(N)$ of $N$ which satisfies $\log(\tau(N))=O( \log N / \log \log N)$. The bound for $\tau(n)$ is Theorem 317 in \cite{Hardy2008AnIT}. We now see that the constant $D_{N,n}/(2^n+1)$ is lower bounded by $N^{-C n\log(n)/\log \log N}$ where $C$ is a universal constant. 
\end{proof}